\documentclass[12pt,twoside,a4paper]{amsart}
\usepackage[T1]{fontenc}
\usepackage{amsmath}
\usepackage{amssymb}
\usepackage{amsthm}
\usepackage{multirow}
\usepackage{centernot}
\usepackage{dsfont}
\usepackage[inner=2.5cm,outer=2.5cm,top=2.5cm,bottom=2.0cm]{geometry}
\usepackage{hyperref}
\usepackage{fancyhdr}
\usepackage{setspace}
\usepackage{graphics}
\usepackage{rotating}
\usepackage{extarrows}
\usepackage{lscape}
\usepackage{rotating}
\usepackage{graphicx}
\usepackage{hyperref}
\usepackage{hyperref}
\hypersetup{
    colorlinks=true,
    linkcolor=blue,
    filecolor=magenta,      
    urlcolor=cyan,
    citecolor=blue}

\newtheorem{thm}{Theorem}[section]
\newtheorem{cor}[thm]{Corollary}
\newtheorem{prop}[thm]{Proposition}

\newtheorem{lem}[thm]{Lemma}
\newtheorem{Def}[thm]{Definition}
\newtheorem{rem}[thm]{Remark}

\newcommand{\be}{\begin{equation}}
\newcommand{\ee}{\end{equation}}
\newcommand{\bee}{\begin{equation*}}
\newcommand{\eee}{\end{equation*}}
\newcommand{\ben}{\begin{enumerate}}
\newcommand{\een}{\end{enumerate}}
\newcommand{\pa}{{\partial}}

\title{On  almost rational Finsler metrics}

\author{Ebtsam H. Taha and Bankteshwar Tiwari}

\address{Harish-Chandra Research Institute, Chhatnag Road, Jhunsi, Allahabad 211019, India}
\address{Department of Mathematics, Faculty of Science, Cairo University, Giza 12613, Egypt}
\email{ebtsam.taha@sci.cu.edu.eg, ebtsamtaha@hri.res.in}
\vspace{-1 cm}
\address{DST-CIMS, Institute of Science, Banaras Hindu University, Varanasi 221005, India}
\email{banktesht@gmail.com}

\begin{document}
\begin{abstract} 
We study a special class of Finsler metrics which we refer to as Almost Rational Finsler metrics (shortly, AR-Finsler metrics). We give necessary and sufficient conditions for an AR-Finsler manifold $(M,F)$ to be Riemannian. The rationality of the associated geometric objects such as Cartan torsion, geodesic spray, Landsberg curvature, $S$-curvature, etc is investigated. We prove for a particular subset of AR-Finsler metrics that if $F$ has isotropic $S$-curvature, then its $S$-curvature identically vanishes. Further, if $F$  has isotropic mean Landsberg curvature, then it is weakly Landsberg. Also, if $F$ is an Einstein metric, then it is Ricci-flat. Moreover, we show that Randers metric can not be AR-Finsler metric. Finally, we provide some examples of AR-Finsler metrics and introduce a new Finsler metric which is called an extended $m$-th root metric. We show under what conditions an extended $m$-th root metric is AR-Finsler metric and study its generalized Kropina change.
\end{abstract}
\maketitle

 
\section{Introduction}
Finsler geometry is a natural generalization of Riemannian geometry. It is wider
in scope and richer in content than Riemannian geometry. A Riemannain metric is quadratic in the fiber coordinates $y$ while a Finsler metric is not necessary  be quadratic in $y$ cf. \cite{Sh}. In the literature of Finsler geometry there are some Finsler structures for which the components of its metric tensor are rational functions in $y$, for example,  Kropina metrics. The rationality in the fiber coordinates $y$ of some geometric objects on Finsler manifold plays a vital role in its characterization. For example, Randers metrics with quadratic Riemann curvature have been investigated in \cite{BLZS}. Finsler spaces with rational spray coefficients have been studied in \cite{spray}. Further, in \cite{$m$-th root}, it was proved that Einstein $m$-th root metrics are Ricci flat using the rationality of the Riemann curvature of an $m$-th root metric. \\

 In this work, we study a class of Finsler metrics, for which the components of its metric tensor $g_{ij}(x,y)$ can be written as a product of a fixed function $\eta(x,y)$ and rational functions $a_{ij}(x,y)$ on the tangent bundle $TM$ (see Definition \ref{Finsler def.}), and call it as Almost Rational Finsler metric, in short AR-Finsler metric. The rationality of the metric tensor forces several geometric objects (like spray coefficients, Riemann curvature, Ricci curvature, Berwald curvature,  etc.)  to have rational functions in $y$. We explicitly calculate some geometric objects of an AR-Finsler metric, viz., Cartan torsion, Cartan mean torsion, spray coefficients, Barthel connection and $S$-curvature. Further, we investigate the conditions for some Finslerian geometric quantities to be rational functions in $y$. We also observe that some geometric objects of an AR-Finsler manifold are always rational and independent of the rationality of the metric tensor.\\
 
 It is interesting to note that there are some famous Finsler metrics  which are  examples of  AR-Finsler metrics. Namely, the generalized Kropina metric, the $m$-th root metric cf.~\cite{$m$-th root}, the Kropina change of the $m$-th root metric and the generalized Kropina change of the $m$-th root metric  \cite{gen. Krop. change, Krop. change}. Besides, we introduce a new Finsler metric looks like the $m$-th root metric but its coefficients are functions on $TM$ rather than functions on $M$, as in the case of $m$-th root metric, which we refer to as \textit{extended $m$-th root metric}. It turns to be an AR-Finsler metric, when its coefficients are rational functions in $y$. We show that its Kropina change as well as generalized Kropina change turn also to be AR-Finsler metric. We further prove that there is no AR-Finsler metric of Randers type. The polynomial $(\alpha, \beta)$-metric has been studied by Z. Shen in \cite{Sh}. In this paper,  we study a special polynomial $(\alpha, \beta)$-metric and find the condition to be an AR-Finsler metric. \\

Some AR-Finsler metrics $F$ are rational in $y$ while others not. We prove the following two  results for irrational $F$ that i) if $(M,F)$ is an AR-Finsler space of isotropic $S$-curvature, then its $S$-curvature identically vanishes; ii) an AR-Finsler space which has isotropic mean Landsberg curvature reduces to weakly Landsberg space. It is an extension of  \cite[Theorem 1.4]{Krop. change}. Further, we show for certain $\eta$ that if $(M,F)$ is an AR-Finsler space of Einstein type, then it has vanishing Ricci curvature. This result represents a generalization of \cite[Theorem 1.1]{$m$-th root} and \cite[Theorem 1.1]{spray}.\\

 In what follows, we give the structure of this paper. Section 2 deals with basic preliminaries required for the rest of this work. In section 3, we define Almost Rational Finsler metric and study some of its characterizations and associated geometric quantities. We give necessary and sufficient for an AR-Finsler manifold to be Riemannian.   We extend the results \cite[Theorem 1.4]{Krop. change},  \cite[Theorem 1.1]{$m$-th root} and \cite[Theorem 1.1]{spray}. Further, we show that no nontrivial Randers metrics is AR-Finsler. In section 4, we give some examples AR-Finsler metrics such as Kropina metric, generalized Kropina metric and $m$-th root metric. Also, we introduce the so called \textit{extended $m$-th root metric} and figure out its generalized Kropina change.

\section{Preliminaries}
Let $M$ be a smooth $n$-dimensional manifold and $TM$ be the corresponding tangent bundle. Let $(x^i)$ be the coordinates of any  point of the base manifold $M$  and $(y^i)$ be a
supporting element at the same point. The  partial differentiation with respect to $x^i$ is denoted by $\partial_i$, while the partial differentiation with respect to  $y^i$ (basis vector fields of the vertical bundle) is denoted by $\dot{\partial}_i$. Let us recall some basics of Finsler geometry. Most of the material presented here with further details  may be found in \cite{Sh}. Hereafter, the Einstein summation convention is in place.

\begin{Def}\label{Finsler def.} A  Finsler structure $F$ on a smooth manifold $M$  is a   mapping
\[F:TM\rightarrow [0, \infty )\]
with the following properties:
\begin{itemize}
    \item[(a)] $F$ is $C^\infty$ on the slit tangent
    bundle  $TM\backslash\{0\}$.

    \item[(b)] $F(x,y)$ is positively homogeneous of first degree in $y$: $F(x,\lambda y) = \lambda F(x,y)$
     for all $y \in TM$ and $\lambda > 0$.

    \item[(c)] The Hessian matrix $g_{ij}(x,y):=\frac{1}{2}\dot{\partial}_i\dot{\partial}_j F^2$
is positive-definite at each point of $TM\backslash\{0\}$.
\end{itemize}
The bilinear symmetric form $g=g_{ij}(x,y)\,dx^i\otimes dx^j$ is called the
Finsler metric tensor or fundamental tensor of the Finsler manifold $(M,F)$.\end{Def} 

Often, a function $F$ satisfies the above mentioned conditions is called a regular Finsler metric. Instead of positive definite, if the metric tensor $g$ is non-degenerate at each point of $TM\backslash\{0\}$, the pair $(M,F)$ is called a pseudo-Finsler manifold.  If $F$ satisfies the conditions (a)-(c)  on an open conic subset $U$ of $TM$ (for every $v \in U$ and $\mu > 0,\, \mu v\in  U $), then $F$ is called  conic Finsler metric. 

\begin{Def}
 The Cartan  torsion associated with a Finsler metric $F$ is given by $C= C_{ijk}\,  dx^i \otimes dx^j \otimes dx^{k}$, where 
$C_{ijk}=\frac{1}{2}\,\dot{\partial}_kg_{ij}=\frac{1}{4}\dot{\partial}_i\dot{\partial}_j\dot{\partial}_k F^2$. The mean Cartan torsion is denoted by $I= I_{k}\, dx^{k},$ where $I_{k}= g^{ij}\,C_{ijk}$.
\end{Def}
 An immediate consequence of the definition of the Cartan tensor is that a Finsler manifold is Riemannian if and only if $C_{ijk}=0$. Deicke proved, in \cite{Deicke}, that a regular Finsler metric is Riemannian if and only if $I_{k}$ vanishes. \\
 
  In 1941, Randers metrics were first studied by the physicist G. Randers,  from the standpoint of general relativity \cite{Randers}. Further, in 1957, R. S. Ingarden applied Randers metrics to the theory of the electron microscope and named them Randers metrics.  A Finsler manifold $(M,F)$ is of Randers type if $F=\alpha + \beta,$ where $\alpha= \sqrt{\alpha_{ij}(x) \,y^{i}\, y^{j}}$  is a Riemannian metric and $\beta = b_{i}(x) \,dx^{i}$ is a $1$-form on $M$ with $|| \beta ||_{\alpha} = \sqrt{\alpha^{ij}(x)\,b_{i}(x)\,b_{j}(x)}< ~1$. The metric $F$ is then said to be a Randers metric. As a generalization of Randers metric, M. Matsumoto introduced $(\alpha, \beta)$-metrics in \cite{Matsumoto}. 
\begin{Def}
 A Finsler metric $F=\alpha\, \phi \left(\frac{\beta}{\alpha}\right)=\alpha\, \phi(s)$  is called $(\alpha, \beta)$-metric if $\phi$ is smooth positive function defined on the interval $(-b_{o}, b_{o})$ such that \[\phi(s)- \phi(s)\, \phi'(s) +(b^2 -s^2 ) \,\phi''(s) > 0,\,\,\,\,\, |s| \leq \, ||\beta||_{\alpha} < b_{o}. \]
\end{Def}

It is known that, the metric tensor associated to $F$ is given by {\Small{
\be \label{gen metric}
g_{ij}= \left(\phi^2 -s \phi\, \phi' \right)\, \alpha_{ij} + \left( \phi \, \phi'' + (\phi')^{2}\right)\, b_{i}\, b_{j}+ \frac{1}{\alpha}\left(\phi\, \phi' -s\left[\phi\,\phi'' + (\phi')^{2} \right] \right) \, \left(b_{i} y_{j} + b_{j} y_{i} -  \frac{s}{\alpha} y_{i} y_{j}\right),
\ee}}
where $y_{j}:= \alpha_{ij}\, y^{i}$.
\begin{Def}
On a Finsler manifold $(M,F)$, the geodesics  are characterized by  
\bee \label{spray coff} \frac{d^2x^i}{dt^2}+ G^i(x,\frac{dx}{dt})=0,
\eee 
where $G^i$ are called spray coefficients of $F$ which defined by 
\be \label{Fspray def}
G^i= \frac{1}{4} \, g^{ir}\left( y^k \, \dot{\pa}_r \pa_k F^2 - \pa_r   F^2 \right) = \frac{1}{2} \, g^{ir}\left( y^k y^s \pa_k g_{rs} -2 \,y^l y^s \pa_{r} g_{ls} \right).
\ee 
\end{Def}
It is worth noting that for a Riemannian metric $F$, the spray coefficients $G^{i}(x,y)$ are quadratic in $y$, whereas for Finsler metric they are highly nonlinear cf.~\cite{oana,Sh}. The geodesic spray induces a nonlinear connection  $ N^i_j:=\dot{\partial}_jG^i$ which is called the Barthel (or Cartan nonlinear) connection associated with $(M,F)$. Thereby, we have the direct sum decomposition 
\[T_{u}(TM)= H_{u}(TM)\oplus V_{u}(TM), \quad\forall u \in TM.\]
The basis vector fields of the horizontal bundle are denoted by $\delta_i:=\partial_i- N^r_i \,\dot{\partial}_r$. 

In addition, $G^i_{jh}:=\dot{\partial}_hN^i_j=\dot{\partial}_h\dot{\partial}_jG^i$ are the coefficients of Berwald connection. Consequently, $G^i_{jhk}:=\dot{\pa}_k G^i_{jh} =\dot{\partial}_k\dot{\partial}_h\dot{\partial}_jG^i$ are the coefficients of Berwald curvature. It is known that, a Finsler manifold with   $G^i_{jhk}=0$ is said to be Berwaldian.  

\begin{Def} 
A Finsler manifold $(M,F)$ is called Landsbergian if Landsberg curvature $L= L_{ijk}\,  dx^i \otimes dx^j \otimes dx^{k}$, where  $L_{ijk}=\dfrac{1}{2}\,y^{m}g_{ms}\,G^s_{ijk}$, vanishes. $(M,F)$ is said to be weakly Landsberg manifold if the mean Landsberg curvature $J= J_{k} \,dx^{k}$ vanishes, where
 \be \label{mean Lands. def.}
 J_{k}= g^{ij}\,L_{ijk}= y^s\, I_{k|s}= y^s\, \pa_s I_k -2 \,G^s\, \dot{\pa}_{s} I_k - N^{s}_{k}\, I_s.
 \ee
 A Finsler manifold $(M,F)$ has isotropic mean Landsberg curvature (or relatively isotropic $J$-curvature) if $J$ can be written in the form 
\be \label{iso mean Lands. def.}
J(x,y)= A(x)\, F(x,y)\, I(x,y), \text{\,\,\,where\,\,\,}  A \in C^{\infty}(M).
\ee
\end{Def}

\begin{Def}
A Finsler metric is called Douglas metric if the Douglas curvature $D = D^i_{jkl}\,\partial_i \otimes dx^j \otimes dx^k \otimes dx^{l}$ vanishes, where
\begin{equation}\label{Douglas}
D^i _{jkl} = G^i _{jkl} - \frac{1}{n+1}\, \dot{\pa}_{j} \dot{\pa}_{k} \dot{\pa}_{l} \left( y^i \,N^s_s  \right).
\end{equation}
\end{Def}

\begin{Def}
The Riemann curvature $R= R^i_k \,{\partial_i}\otimes dx^k$ of a Finsler space is defined by
\be \label{R curv def}
R^i_k=2\,\partial_k G^i - y^{j}\,\partial_j\,\dot{\partial}_k G^i +2 G^j\,\dot{\partial}_k \,\dot{\partial}_j G^i - \dot{\partial}_j G^i\, \dot{\partial}_k G^j.
\ee
The Weyl projective curvature  $W= W^{i}_{j}\, \pa_{i} \otimes dx^j$ is given by
\be \label{Wely cur}
W^{i}_{j} = Q^{i}_{j} -\frac{1}{n+1}\, y^i \dot{\pa}_{s} Q^{i}_{s},
\ee 
where $Q^{i}_{j}= R^{i}_{j} -\dfrac{Ric}{n+1}\, \delta^{i}_{j}  \, \text{ and } \, Ric=R^m_m .$ Further, the $\chi$-curvature is defined by \be \label{chi cur}\chi_l= -\frac{1}{6}\, \Big \{ 2\, \dot{\pa}_{i}R^i_{l} + \dot{\pa}_{l} Ric \Big \}.
\ee

\end{Def}
\begin{Def}
Given a volume form $dV:= \sigma(x)\, dx^1 \wedge...\wedge dx^n$ on $(M,F)$, for any $x\in M$ and $y\in T_xM\setminus\left\lbrace 0 \right\rbrace $,  the distortion $\tau $ of the Finsler metric $F$ is defined by $\tau (x,y)=\ln \left(\dfrac{\sqrt{\det( g_{ij}(x,y))}}{\sigma(x)}\right)$. The derivative of   $\tau$ along  a geodesic  $\gamma$ with $\gamma(0)=x$ and $\gamma'(0)=y$ is  called the $S$-curvature at $x$ along $y$. More explicitly, the $S$-curvature is defined by
\begin{equation*}
S(x,y)=\frac{d}{dt}\left[ \tau(\gamma(t),\gamma'(t))\right] _{t=0}.
\end{equation*}
\end{Def}
 For a local coordinate system in $(M,F)$, the $S$-curvature has the following expression
\begin{equation}\label{eqn8}
S(x,y)=N^m_m (x, y)-y^m\, \partial_{m}\log\left( \sigma(x)\right).
\end{equation}
 Another non-Riemannian object, manifested from the  differentiation of $S$-curvature, is called \textit{Berwald mean curvature} (or \textit{$E$-curvature}) $E= E_{ij} \,dx^i \otimes dx^j ,$  where 
 \be \label{E-cur}
 E_{ij}=\frac{1}{2} \dot{\pa}_{i}\dot{\pa}_{j}S=\frac{1}{2}\dot{\pa}_{i}\dot{\pa}_{j}N^{m}_{m} .\ee

\section{Almost rational Finsler metrics}
\begin{Def}
A Finsler metric $F$ on a manifold $M$ is called  Almost Rational Finsler (simply, AR-Finsler) metric if the coefficients $g_{ij}(x,y)$ of its metric tensor can be expressed as follows:
\be \label{gdef}
g_{ij}(x,y)= \eta(x,y)\, a_{ij}(x,y),
\ee
where
\begin{itemize}
\item[(a)] $\eta: TM \longrightarrow [0,\infty )$ is a smooth function,
\item[(b)] the matrix $( a_{ij}(x,y))_{1 \leq i,j \leq n}$ is  symmetric positive definite  with each $ a_{ij}(x,y)$ be a rational function in the fiber coordinate $y$,
\item[(c)] $\{\eta(x,y)\, a_{ij}(x,y)\}$ is positive homogeneous of degree zero in $y$ for all $1 \leq i,j \leq n$.
\end{itemize}
The pair $(M,F)$ is said to be  AR-Finsler manifold. In addition if $\eta$ is a rational function in $y$, we call $F$ as rational Finsler metric and the pair $(M,F)$ as  rational Finsler structure.
\end{Def}
If we relax condition (b) as the matrix $(a_{ij}(x,y))_{1 \leq i,j \leq n}$ is nondegenerate (respectively, degenerate)  then, we  deal with\textit{ pseudo} or \textit{nondegenerate (respectively, degenerate) AR-Finsler metric}.
\begin{rem}
It is clear that from expression $\eqref{gdef}$, the  metric tensor $g_{ij}(x,y)$ are rational functions in $y$ if and only if $\eta$ is a rational function in $y$. For example, in \cite[Examples 3, 4]{semi} the function $\eta$  is a rational function in $y$.
\end{rem}
\begin{lem}\label{prelem}
Let $(M,F)$ be an AR-Finsler manifold. Then we have the following:
\begin{itemize}
\item[(1)] The function $\frac{F^{2}}{\eta} (x,y)$ is rational in $y$,
\item[(2)]The inverse metric tensor of $F$ is given by $g^{ij}(x,y)= \frac{a^{ij}(x,y)}{\eta(x,y)},$
\item[(3)] The expression $\{ \eta\, \dot{\pa}_i a_{jk} + a_{jk}\, \dot{\pa}_i \eta \} $ is symmetric in the indices $i,j,k$.
\end{itemize}
\end{lem}
\begin{proof}
$(1)$ follows directly from the fact that the multiplication of rational functions always results a rational function and $\frac{F^{2}}{\eta} (x,y)= a_{ij}(x,y)\, y^i\, y^j$. $(2)$ is straight forward from the fact $g_{ij}\,g^{ik}= \delta^k_j$. 
$(3)$ follows by plugging $\eqref{gdef}$ into the Cartan torsion formula $C_{ijk}=\frac{1}{2}\dot{\partial}_kg_{ij}$, which is totally symmetric in its indices.
\end{proof}

\begin{prop}\label{pa dot log eta rat}
Let $(M,F)$ be an $n$-dimensional AR-Finsler manifold, $n\geq 2$. Then, the function $\dot{\pa}_{i} \log(\eta(x,y))$ is a rational function in $y$ for all $i=1,...,n$.

\end{prop}
\begin{proof}
It is clear that $2C_{ijk}=2 C_{kij}$. That is, in the view of Lemma \ref{prelem}~(3),
\begin{eqnarray*}
\eta\, \dot{\pa}_i a_{jk} + a_{jk}\, \dot{\pa}_i \eta &=& \eta\, \dot{\pa}_k a_{ij} + a_{ij}\, \dot{\pa}_k \eta\,\, \Longleftrightarrow \\
 \eta \left(\dot{\pa}_i a_{jk} - \dot{\pa}_k a_{ij}  \right) &=& -a_{jk}\, \dot{\pa}_i \eta + a_{ij}\, \dot{\pa}_k \eta.
\end{eqnarray*} 
Multiplying both sides of the previous relation by $a^{ji}$, we get
 \[
 \eta \, a^{ji}\left(\dot{\pa}_i a_{jk} - \dot{\pa}_k a_{ij}  \right) = -\delta^i_k\, \dot{\pa}_i \eta + n\, \dot{\pa}_k \eta = (n-1)\,\dot{\pa}_k \eta.\]
Thus, we have
\be \label{dot log eta rat}
\dot{\pa}_k \log(\eta) = \frac{1}{n-1} \,a^{ji}\,\left(\dot{\pa}_i a_{jk} - \dot{\pa}_k a_{ij}\right). 
\ee
Therefore, $\dot{\pa}_k \log\left(\eta(x,y)\right)$ are rational functions in $y$.
\end{proof}
It is known that the partial derivatives $\dot{\pa}_k$ and  $\pa_{l}$ commutes. Also, the differential of any rational function $f(x,y)$ in $y$ with respect to $x$ remains rational in $y$, thereby, we have:
\begin{cor}\label{commute}
The function $\dot{\pa}_k \pa_{k_1} \pa_{k_2}... \pa_{k_r}\log\left(\eta(x,y)\right)$ is a rational function in $y$ for all $1 \leq k ,k_{1},...,k_{r}\leq n,\,\, r \in \mathbb{N}$.
\end{cor}
\begin{rem}
Given an AR-Finsler manifold, the rationality of the associated  Cartan torsion $C_{ijk}(x,y)$  in $y$ depends on the rationality of the function $\eta$ in $y$. This results from the expression \[C_{ijk} =  \eta\, \dot{\pa}_i a_{jk} + a_{jk}\, \dot{\pa}_i \eta.\] In other words, $C_{ijk}(x,y)$  are rational functions in $y$ if and only if $\eta$ is rational  function in $y$. However, the following proposition shows that rationality of the mean Cartan torsion of an AR-Finsler manifold is independent of the rationality of the function $\eta$.
\end{rem}

\begin{prop}\label{main scalar rat}
Let $(M,F)$ be an AR-Finsler manifold. Then, the mean Cartan torsion  is rational  in $y$.
\end{prop}
\begin{proof}
 By Lemma~\ref{prelem}~(2), we get
\be \label{main saclar formula}
I_{i}=g^{jk}\,C_{ijk}= \frac{a^{jk}}{\eta} \left( \eta\, \dot{\pa}_i a_{jk} + a_{jk}\, \dot{\pa}_i \eta \right)= a^{jk}  \dot{\pa}_i a_{jk} + n \, \dot{\pa}_i \log(\eta) .
\ee
  Therefore, the proof is completed by the use of Proposition \ref{pa dot log eta rat}.
\end{proof}
In the view of formulae \eqref{dot log eta rat} and \eqref{main saclar formula}, we have the following result.
\begin{cor}
Let $(M,F)$ be an $n$-dimensional AR-Finsler manifold, $n \geq 2$. Then,  the mean Cartan torsion $I_k$ can be expressed as follows:
\[I_k = \frac{1}{n-1}\, a^{rs} \left( n\, \dot{\pa}_r a_{sk} - \dot{\pa}_k a_{rs}  \right). \]
\end{cor}
\begin{prop}
Let $(M,F)$ be an $n$-dimensional AR-Finsler manifold. Then, we have
\be \label{delta log eta}
 \delta_{k} \log\left(\eta\right) = - \frac{1}{n}\, a^{ij}\, a_{ij|k}, \ee where  $|$ is the horizontal covariant derivative with respect to the Berwald connection.
\end{prop}
\begin{proof}
It is known that $g_{ij|k}=0$. Thus, by $\eqref{gdef}$, we obtain
\[0= (\eta \, a_{ij})_{|k} = \eta _{|k}\, a_{ij} + \eta \,a_{ij|k} = (\delta_{k} \eta)\, a_{ij} + \eta \,a_{ij|k} \Longleftrightarrow 
(\delta_{k} \eta)\, a_{ij} =- \eta \,a_{ij|k}. \]
Then, we have
 \[a_{ij}\, \delta_{k}\log(\eta) = -a_{ij|k}.\]
 Hence, the proof is completed by multiplying both sides of the last relation by $a^{ij}$ and taking into account $a^{ij}\, a_{ij}=n$.
\end{proof}
\begin{thm}\label{Riem}
Let $(M,F)$ be a regular AR-Finsler manifold. Then, $(M,F)$ is Riemannian if and only if the functions $\dot{\pa}_i \log(\eta)$ have the following form:
\be
\dot{\pa}_i \log(\eta) = -\frac{1}{n} \,a^{jk}\,\dot{\pa}_i a_{jk}  
\ee
\end{thm}
\begin{proof}
Suppose $(M,F)$ is Riemannian. 
Thus, $C_{ijk}=0$, this means that, 
\begin{eqnarray*}
0=\eta\, \dot{\pa}_i a_{jk} + a_{jk}\, \dot{\pa}_i \eta \Longrightarrow   a_{jk}\, \dot{\pa}_i \log(\eta) = - \dot{\pa}_i a_{jk} \Longrightarrow n \, \dot{\pa}_i \log(\eta) = -  a^{jk}\, \dot{\pa}_i a_{jk}.
\end{eqnarray*} 
For the converse, assume that $\dot{\pa}_i \log(\eta) = -\frac{1}{n} \,a^{jk}\,\dot{\pa}_i a_{jk}$. Then, we get 
\begin{eqnarray*}
 n \, \dot{\pa}_i \log(\eta) = -  a^{jk}\, \dot{\pa}_i a_{jk} \Longrightarrow   n \, \dot{\pa}_i \eta = -\eta\, a^{jk}\, \dot{\pa}_i a_{jk} \Longrightarrow a^{jk}\,a_{jk}\,\dot{\pa}_i \eta = -\eta\, a^{jk}\, \dot{\pa}_i a_{jk}
 && \\
 \Longleftrightarrow a^{jk}\,( \eta\, \dot{\pa}_i a_{jk} + a_{jk}\, \dot{\pa}_i \eta )=0\Longleftrightarrow  a^{jk}\,C_{ijk}=0.
 \end{eqnarray*}
 In view of Lemma \ref{prelem} (2), \[a^{jk}\,C_{ijk}=0 \Longrightarrow \frac{1}{\eta}\, a^{jk}\,C_{ijk}=0 \Longleftrightarrow g^{jk}\,C_{ijk}=0 \Longleftrightarrow I_{i}=0.\]
Therefore, $(M,F)$ is Riemannian space by Deicke's theorem \cite{Deicke}.
\end{proof}
\begin{cor}
Let $(M,F)$ be a regular AR-Finsler manifold,  $n \geq 2$. Then, the necessary and sufficient condition for $(M,F)$ to be Riemannian is
\be 
a^{jk}\,\left( n\,\dot{\pa}_k a_{ji} - \dot{\pa}_i a_{jk}\right)=0.
\ee
\end{cor} 
\begin{proof}
By making use of Theorem \ref{Riem} and formula \eqref{dot log eta rat}, we get
\[\dot{\pa}_i \log(\eta) = -\frac{1}{n} \,a^{jk}\,\dot{\pa}_i a_{jk} = \frac{1}{n-1} \,a^{jk}\,\left(\dot{\pa}_k a_{ji} - \dot{\pa}_i a_{jk} \right), \] which is equivalent to
{\small{\[\frac{1}{n-1} \,a^{jk}\, \dot{\pa}_k a_{ji} + \left(\frac{1}{n} -\frac{1}{n-1} \right)a^{jk}\, \dot{\pa}_i a_{jk} =0 \Longleftrightarrow  \frac{1}{n(n-1)} \,a^{jk}\,\left( n\dot{\pa}_k a_{ji} - \dot{\pa}_i a_{jk}\right)=0.\]}}
\vspace*{-1.1cm}\[\qedhere\]
\end{proof}

\begin{prop}\label{G eta rat}
Let $(M,F)$ be an AR-Finsler manifold. Then, the geodesic spray coefficients $G^i(x,y)$ of $F$ are rational functions in $y$ if and only if $\pa_k\log\left(\eta(x,y)\right)$ are rational functions in $y$.
\end{prop}
\begin{proof}
Plug expression \eqref{gdef} into formula \eqref{Fspray def}, we obtain
\be \label{Fspray eta}
G^i = \frac{1}{2} \, \left( y^i y^l -\frac{1}{2}\, y^r y^s \, a_{rs} a^{il} \right)\pa_l\log(\eta) +  \frac{1}{2} \,y^k\, y^s \, a^{li} \left( \pa_{k} a_{ls} - \frac{1}{2} \, \pa_{l}a_{ks} \right).
\ee
Assume that $\pa_k\log\left(\eta(x,y)\right)$ are rational functions in $y$. The proof is completed by noting that the multiplication of rational functions is rational and  the partial derivative  $\pa_{l}$ the rational functions $a_{ij}(x,y)$ with respect to $x$ remains rational in $y$. The converse follows directly from \eqref{Fspray eta}.
\end{proof}

\begin{prop}\label{N eta rat}
Let $(M,F)$ be an AR-Finsler manifold. Then, the coefficients of Barthel connection $N^i_{j}$ are given by
\begin{eqnarray*}
N^j_{i}&=& \frac{1}{2} \left[ y^j\, \pa_i\log(\eta)  + \delta^{j}_{i}\,y^k\, \pa_k\log(\eta) \right] \\
&&- \frac{1}{4}\, y^r\, \left( 2\, a_{ir}\,a^{jl} + y^s [a^{jl} \dot{\pa}_{i} a_{rs} + a_{rs}\,\dot{\pa}_{i} a^{jl}] \right)\, \pa_l\log(\eta)  \\
&&
+\frac{1}{2} \left( y^j \, y^l - \frac{1}{2} \, y^r \, y^s\, a_{rs}\, a^{jl} \right) \dot{\pa}_{i} {\pa}_{l}\log(\eta) \\
&&
+  \frac{1}{2}  \dot{\pa}_{i} \left(\,y^k\, y^s \, a^{lj} \left[ \pa_{k} a_{ls} - \frac{1}{2} \, \pa_{l}a_{ks} \right]\right).
\end{eqnarray*}
Consequently, $N^i_{j}(x,y)$ are rational functions in $y$ if and only if the functions  $\pa_k\log\left(\eta(x,y)\right)$ are  rational  in $y$.
\end{prop}
\begin{proof}
The expression of $N^j_i (x,y)$ immediately follows by differentiating formula \eqref{Fspray eta} with respect to $y^j$. The necessary and sufficient condition for  $N^j_i (x,y)$ to be rational functions in $y$ is a direct consequence of Corollary \ref{commute} and Proposition \ref{G eta rat} along with using the fact that the multiplication of rational functions is rational and the partial derivative of a rational function in $y$ with respect to $x$ remains rational in $y$. 
\end{proof}
\begin{cor}
Let $(M,F)$ be an $n$-dimensional AR-Finsler  manifold, $n \geq 2$. Then, the functions $\delta_{k}\log\left(\eta(x,y)\right)$ are rational in $y$ if and only if  $\pa_i\log\left(\eta(x,y)\right)$ are rational functions in $y$.
\end{cor}
\begin{proof}
Since $\delta_{k}\log(\eta) = \pa_{k}\log(\eta)- N^{r}_{k} \, \dot{\pa}_{r} \log(\eta)$, thus, the proof is completed by making use of Propositions \ref{pa dot log eta rat} and  \ref{N eta rat}  together with using the fact the multiplication of rational functions is rational.
\end{proof}
A direct consequence of Propositions \ref{G eta rat} and \ref{N eta rat} is the following result. 
\begin{cor}\label{Berwald}
Let $(M,F)$ be an AR-Finsler manifold. Then, the Berwald connection $G^i _{jk}$ and   Berwald curvature $G^i _{jkl}$ are rational functions in $y$ if and only if  the functions $\pa_i\log \left(\eta(x,y)\right)$ are rational in $y$. 
\end{cor}

\begin{prop}
Let $(M,F,dV)$ be an AR-Finsler manifold equipped with an arbitrary volume form. Then, the $S$-curvature is given by
\begin{eqnarray*}
S &=& \frac{1}{2} \left( n\, y^{l}  -\frac{1}{2} y^{r}\, y^{s}\, \{ a^{ml}\, \dot{\pa}_{m}a_{rs} + a_{rs} \dot{\pa}_{m} a^{ml}  \} \right)\,\pa_{l}\log(\eta) 
 \\ &&
+\frac{1}{2}  \left( y^{m} y^{l} -  \frac{1}{2}  y^{r} y^{s}  a_{rs} a^{ml}\right)\dot{\pa}_{m}\pa_{l}\log(\eta)+ \frac{1}{2} y^{k} a^{ml}\left( \pa_{m} a_{lk} + \pa_{k} a_{lm} -\frac{1}{4} \pa_{l} a_{mk}\right) \\ &&
+ \frac{1}{2} y^{k} y^{s} \left[ \dot{\pa}_{m} a^{ml} ( \pa_{k} a_{ls} -\frac{1}{2} \pa_{l} a_{ks}) + a^{ml} \left( \dot{\pa}_{m}\pa_{k} a_{ls}- \frac{1}{2} \dot{\pa}_{m}\pa_{l} a_{ks}\right) \right] -  y^{m}\pa_{m} \log( \sigma(x)).
\end{eqnarray*}
Thus, $S$-curvature is rational function in $y$ if and only if  the functions $\pa_i\log\left(\eta(x,y)\right)$ are rational in $y$. Consequently, the $E$-curvature is rational in $y$ if and only if the functions $\pa_i\log\left(\eta(x,y)\right)$ are rational  in $y$.
\end{prop}
\begin{proof}
It follows from Propositions \ref{N eta rat} together with formulae \eqref{eqn8} and \eqref{E-cur} along with using the fact that the multiplication of rational functions is rational.  
\end{proof}
It is worth mentioning that some Finsler structures $F$ are rational in $y$ while others are not. For example, Kropina metric is rational in $y$, generalized Kropina metric $F(x,y)=\frac{\alpha^{k+1}}{\beta^k}$, where $k$ is an even number, is not rational function in $y$ (see section 4, for details). The following result  holds only when a Finsler metric $F$ is not rational function in $y$.
\begin{thm}
Let $(M,F, dV)$ be an AR-Finsler manifold equipped with an arbitrary volume form such that $F(x,y)$ is not rational in $y$ and the functions $\pa_i\log\left(\eta(x,y)\right)$ are rational in $y$. If $(M,F,dV)$ has isotropic $S$-curvature, then $(M,F)$ has vanishing $S$-curvature. 
\end{thm}
\begin{proof}
A Finsler metric $F$ is said to have isotopic $S$-curvature if there exists a function $A$ in $M$ such that \[S(x,y)= (n+1)\, A(x)\, F(x,y).\]
Assume that the functions $\pa_i\log\left(\eta(x,y)\right)$ are rational in $y$. Then, by Proposition \ref{main scalar rat}, the $S$-curvature is  rational in $y$, on the other hand $F(x,y)$ is not rational in $y$ by hypothesis. Therefore, $S$ identically vanishes.
\end{proof}
 \begin{prop}
 Let $(M,F)$ be an AR-Finsler manifold. Then, the Douglas curvature components $D^i _{jkl}$ are rational functions in $y$ if and only if the functions $\pa_i\log\left(\eta(x,y)\right)$ are rational in $y$.
 \end{prop}
\begin{proof}
It results from Proposition \ref{N eta rat} and Corollary \ref{Berwald} together with formula \eqref{Douglas}.
\end{proof}
\begin{prop}
Let $(M,F)$ be an AR-Finsler manifold, $n \geq 2$. Then, the Landsberg curvature components $L_{ijk}$ are rational functions in $y$ if and only if  $\pa_i\log\left(\eta(x,y)\right)$ are rational functions in $y$.
\end{prop}
\begin{proof}
It results by making use of Corollary \ref{Berwald} along with the formula \[L_{ijk}=\dfrac{1}{2}\,y^{m}g_{ms}\,G^s_{ijk}.\]
\vspace*{-1cm}\[\qedhere\] 
\end{proof}

\begin{lem}\label{mean Landsberg eta rat}
Let $(M,F)$ be an AR-Finsler manifold with $n \geq 2$. Then, the mean Landsberg curvature $J_{k}(x,y)$  are rational functions in $y$ if and only if $\pa_k\log\left(\eta(x,y)\right)$  are rational functions in $y$.
\end{lem}
\begin{proof}
Proposition \ref{main scalar rat} reads that the functions $I_k (x,y)$ are rational in $y$. Thereby,  $\pa_{i}{I}_k (x,y)$ and $\dot{\pa}_{i}{I}_k (x,y)$ are  rational functions in $y$. Propositions \ref{G eta rat} and \ref{N eta rat} say that the rationally of the quantities $G^{s}(x,y)$ and $N^{s}_{k}(x,y)$ in $y$ is equivalent to the rationally of the functions   $\pa_k\log\left(\eta(x,y)\right)$ in $y$. Thus, by formula \eqref{mean Lands. def.}, the proof is completed.
\end{proof}

\begin{thm}
Let $(M,F)$ be an AR-Finsler manifold with $n \geq 2$ such that $F(x,y)$ is not a rational function in $y$ and the functions $\pa_i\log\left(\eta(x,y)\right)$ are rational in $y$. If $(M,F)$ has isotropic mean Landsberg curvature, then $(M,F)$ is weakly Landsberg manifold. 
\end{thm}
\begin{proof}
It is clear from Lemma \ref{mean Landsberg eta rat} that $J_{k}(x,y)$ are rational functions in $y$. Suppose $(M,F)$ is an isotropic mean Landsberg manifold, that is $J$ can be written as in \eqref{iso mean Lands. def.}. But $F$ is  not rational function in $y$, by our assumption, and $I$ is rational in $y$ by Proposition~\ref{main scalar rat}. Therefore, by formula \eqref{iso mean Lands. def.}, $J$ must identically vanishes. Thus, $(M,F)$ is a weakly Landsberg manifold.
\end{proof}

\begin{prop}\label{Riem cur. prop.}
Let $(M,F)$ be an AR-Finsler manifold, $n \geq 2$. Then, the  Riemannian curvature is a rational function in $y$ if and only if  the functions $\pa_i\log\left(\eta(x,y)\right)$ are rational in $y$ for. Consequently, the Weyl curvature and $\chi$-curvature are rational in $y$ if and only if  $\pa_i\log\left(\eta(x,y)\right)$ are rational functions in $y$.
\end{prop}
\begin{proof}Let $F$ be an AR-Finsler metric such that $\pa_i\log\left(\eta(x,y)\right)$  are rational functions in $y$. 
It is clear from Propositions \ref{G eta rat} that $G^i (x,y) $ are rational functions in $y$. Thus, $\partial_k G^i (x,y) $ are rational functions in $y$ as the differentiation with respect to the manifold coordinates $\partial_k$ is not affecting the rationality of $G^i (x,y) $ in $y$. Also the product of rational functions is a rational function, considering \eqref{R curv def} leads to the rationality of $R^i_k(x,y)$ in $y$.
 Similarly, by formulae  \eqref{Wely cur} and \eqref{chi cur} the proof is completed. 
\end{proof}
The following result holds only in case of the AR-Finsler metrics with $\eta$ is not a rational function in $y$. It may be considered as a generalization of \cite[Theorem 1.1]{$m$-th root} and \cite[Theorem 1.1]{spray}.
\begin{thm}Let $(M,F)$ be an AR-Finsler manifold $n \geq 2$ such that the function $\eta$ is not rational in $y$ and the functions $\pa_i\log\left(\eta(x,y)\right)$ are rational in $y$. If $F$ is an Einstein metric then  it is Ricci-flat. \end{thm}
\begin{proof}
By Propositions \ref{Riem cur. prop.}, $Ric=R^{i}_{i}$ is a rational function in $y$.  By our postulates, $F$ is an Einstein metric, that is, \[Ric = (n-1)\, K\, F^{2}, \text{ where } K \in C^{\infty}(M).\] The function $\eta$ is not rational in $y$ thereby, $F^{2}= \eta \, a_{ij}\, y^{i} \, y^{j}$ is not  rational function in $y$. Hence, $K=0$ which means $Ric =0$.
\end{proof}
\begin{rem}
It may be remarked here that we can not expect all irrational Finsler metrics $g_{ij}$ in the form of \eqref{gdef}. For example, Shen's circles with radius $1$ and centered at $(0,0)$ cf. \cite{BLZS}.
Let $M= \mathbb{R}^{2}$ equipped with the following Finsler metric of Randers type
\[F(x,y)=F(x^1 , x^2; y^1 ,y^2)= \sqrt{(y^1)^2 + (y^2)^2} + A( x^1 , x^2)\,(y^1 +  y^2) ,\]
where $A( x^1 , x^2) \in C^{\infty}(M)$. It is easy to see that, the associated metric tensor has the following components:
\[ g_{11}= 1+ A^{2}( x^1 , x^2) +A( x^1 , x^2)\,\frac{(y^1)^3 + (y^2)^3 + y^1 \, (y^2)^{2}}{[(y^1)^2 + (y^2)^2]^\frac{3}{2}}, \]
\[ g_{22}= 1+ A^{2}( x^1 , x^2)  +A( x^1 , x^2)\,\frac{(y^1)^3 + (y^2)^3 + y^2 \, (y^1)^{2}}{[(y^1)^2 + (y^2)^2]^\frac{3}{2}}, \]
\[ g_{12}=  A^{2}( x^1 , x^2)+ A( x^1 , x^2)\,\frac{(y^1)^3 + (y^2)^3}{[(y^1)^2 + (y^2)^2]^\frac{3}{2}}.\] This metric has irrational spray coefficients \[G^{1}= \frac{y^2}{2} \sqrt{(y^1)^2 + (y^2)^2}, \qquad G^{2}= \frac{y^1}{2} \sqrt{(y^1)^2 + (y^2)^2}.\] 
Further, its Cartan torsion is irrational in $y$ while its Riemannian curvature is quadratic.
\end{rem}
The  Shen's circles example motivates us to study the case of Randers metrics which leads to the following result.
\begin{thm}
  There is no AR-Finsler manifold of Randers type.
\end{thm}
\begin{proof}
Suppose that $F=\alpha + \beta$. Thus, its metric tensor is given by
\begin{eqnarray*}\label{Randers metric}
g_{ij}(x,y)&=&\frac{1}{2}\dot{\partial}_i\dot{\partial}_j F^2(x,y)= \frac{1}{2}\dot{\partial}_i\dot{\partial}_j (\alpha(x,y) + \beta(x,y))^2 \\ &=& \frac{1}{2}\dot{\partial}_i\dot{\partial}_j \alpha^2(x,y) + \frac{1}{2}\dot{\partial}_i\dot{\partial}_j \beta^2(x,y) + \dot{\partial}_i\dot{\partial}_j (\alpha(x,y) \, \beta(x,y)) \\
 &=&\alpha_{ij}(x) + b_{i}(x)\, b_{j}(x)+ \frac{1}{\alpha}\left\{\left(\alpha_{ij}(x)- \frac{y_{i}\,y_{j}}{\alpha^2(x,y)} \right)\beta(x,y) + b_{i}(x)\,\,y_{j} + b_{j}(x)\,\,y_{i}\right\},
\end{eqnarray*}
where $y_{i}= \alpha_{ij}(x)\,y^{j}$.
It is clear that $\alpha_{ij} + b_{i}\, b_{j}$ is a function of $x$ only, thus\[\left\{\left(\alpha_{ij}(x) - \frac{y_{i}\,y_{j}}{\alpha^2} \right)\beta(x,y) + b_{i}(x)\,\,y_{j} + b_{j}(x)\,\,y_{i}\right\}\] is a rational function in $y$ and $\frac{1}{\alpha(x,y)}$ is not rational function in $y$.
Therefore, $g_{ij}$ can be written in the form of \eqref{gdef} with \[\eta(x,y) = \frac{1}{\alpha(x,y)}\,\, \text{ and }\,\, a_{ij}(x,y)= \left\{\left(\alpha_{ij}(x) - \frac{y_{i}\,y_{j}}{\alpha^2(x,y)} \right)\beta(x,y) + b_{i}(x)\,\,y_{j} + b_{j}(x)\,\,y_{i}\right\} \]if and only if  $\left\{\alpha_{ij}(x) + b_{i}(x)\, b_{j}(x)\right\}$ vanish identically, which is impossible. Hence, there is no Randers structure whose fundamental metric can be written in the form of $\eqref{gdef}$.
\end{proof}
\begin{rem}
It should be noted that, Randers metrics of Berwald type cf.~\cite{Sh} (in which $\beta$ is parallel with respect to $\alpha$) are examples of non AR-Finsler metrics with rational, more precisely quadratic, spray coefficients.
\end{rem} 
\section{Examples of AR-Finsler metrics}
In this section we study some examples of AR-Finsler metrics. \\


\textbf{The generalized Kropina metric.}\\

The Finsler function of generalized Kropina metric  is defined by $F= \alpha\,\phi(s)$ where $\phi(s)= \frac{1}{s^m},\,\, m\neq 0,-1.$  By substituting in formula \eqref{gen metric}, the metric tensor is given by 
\be \label{gen Krop.}
g_{ij}= \frac{m+1}{s^{2m}}\, \alpha_{ij} + \frac{m(2m+1)}{s^{2m+2}} \, b_{i}\, b_{j} - \frac{m(m+2)}{\alpha \,s^{2m+1}} \, \left(b_{i} y_{j} + b_{j} y_{i}  -  \frac{s}{\alpha} y_{i} y_{j} \right). \ee
\\
It is clear that, the function $\frac{1}{s^{2m+2}} = \left( \frac{\alpha^2}{\beta^2} \right)^{m+1}= \left( \frac{\alpha_{ij}(x)\, y^{i}\, y^{j}}{b_{i}(x)\, b_{j}(x)\, y^{i}\, y^{j}} \right)^{m+1}$ is a rational function in $y$ when $m$ is integer and irrational function in $y$ when $m$ not integer. Similarly, the functions $\frac{1}{s^{2m}}= \left( \frac{\alpha_{ij}(x)\, y^{i}\, y^{j}}{b_{i}(x)\, b_{j}(x)\, y^{i}\, y^{j}} \right)^{m}$ and $ \frac{1}{\alpha \,s^{2m+1}}= \frac{1}{ b_{r}(x)\, y^{r}}\left( \frac{\alpha_{ij}(x)\, y^{i}\, y^{j}}{b_{i}(x)\, b_{j}(x)\, y^{i}\, y^{j}} \right)^{m}$ are rational functions in $y$ when $m$ is integer and irrational functions in $y$ when $m$ not integer. Which leads to the metric \eqref{gen Krop.} is rational in $y$ when $m$ is integer. Further, expression \eqref{gen Krop.} can be written as follows: 
\[g_{ij}= \frac{1}{s^{2m}}\,\left[(m+1)\, \alpha_{ij} + \frac{m(2m+1)}{s^2} \, b_{i}\, b_{j} - \frac{m(m+2)}{\alpha \,s} \, \left(b_{i} y_{j} + y_{i} b_{j} -  \frac{s}{\alpha} y_{i} y_{j} \right) \right]. \] 
Thereby, expression \eqref{gen Krop.} has the form of \eqref{gdef} with
\[ \eta = \frac{1}{s^{2m}} \, \text{ and } a_{ij}= (m+1)\, \alpha_{ij} + \frac{m(2m+1)}{s^2} \, b_{i}\, b_{j} - \frac{m(m+2)}{\alpha \,s} \, \left(b_{i} y_{j} + y_{i} b_{j} -  \frac{s}{\alpha} y_{i} y_{j} \right).  \]

It should be noted that, the functions $a_{ij}(x,y)$ are rational in $y$. Hence, the generalized Kropina metric is an AR-Finsler metric.  By direct calculations, we get \[\dot{\pa}_{r} \log(\eta(x,y))= -2m\, \frac{b_{r}(x)}{b_{i}(x)\, y^{i}} +2m \frac{\alpha_{ir}(x)\, y^{i}}{\alpha_{ij}(x)\, y^{i}\, y^{j}} \] which are rational functions in $y$ as expected from Proposition \ref{pa dot log eta rat}. Similarly, the functions 
\[\pa_{r} \log(\eta(x,y))=  -2m\, \frac{y^{i}\pa_{r} b_{i}(x)}{b_{i}(x)\, y^{i}} +m \frac{ y^{i}\,y^{j} \, \pa_{r} \alpha_{ij}(x)}{\alpha_{ij}(x)\, y^{i}\, y^{j}} \] are  rational functions in $y$.  \\ 
\begin{rem} In particular for $k=1$, $F$ is the well known Kropina  metric, which  has been studied in details in \cite{Matsumoto, Sh}.
\end{rem}

\textbf{The special polynomial  $(\alpha , \beta)$-metric.}\\

Let $F:=\alpha\,\phi(s)$ where $\phi(s)= a \,s^k + b\,s^m, \,\, m=k (\text{mod}) 2\,\,$  such that
\[ a \,(m-1)\, s^m + b \,(k-1)\,  s^k > 0,\,\,\ |s|< b_{o}.\]

By straight forward calculations, we get
\begin{eqnarray*}
g_{ij} &=& -(a \,s^m + b \,s^k) \{a \,(m-1)\, s^m + b \,(k-1)\,  s^k \}\, \alpha_{ij}   \\
&& + \frac{a^2 \, m\,(2 m-1)\, s^{2m} + b^2 \,k\,(2 k-1) \, s^{2k} + a\, b\, (m+k-1) \,(m+n)\, s^{m+k}}{s^2}\,\,b_{i}\, b_{j} \\
&& - \frac{2\,a^2 \, m\,( m-1)\, s^{2m}  + b^2 \,k\,(k-1)\,s^{2k}  + a\, b (m+k-2) \,(m+k) \,s^{m+k}} {s \,\alpha} \times
\\ && \left(b_{i} y_{j} + y_{i} b_{j} -  \frac{s}{\alpha} y_{i} y_{j}\right).
\end{eqnarray*}
Thereby, $g_{ij}$ can be written in the  form:
\begin{eqnarray*}
g_{ij} &=& -\left( \frac{O}{2m} \,s^{2m} + \frac{Z}{k}\,s^{2k} + \frac{U}{(m+k)} \, s^{m+k} \right)\, \alpha_{ij}   \\
&& + \,\,\frac{A\, s^{2m} +B\, s^{2k} + H\, s^{m+k}}{s^2 }\,\,b_{i}\, b_{j} \\
&& -\,\, \frac{O\, s^{2m}  + Z\,s^{2k}  + U \,s^{m+k}} {s \,\alpha} \,\left(b_{i} y_{j} + y_{i} b_{j} -  \frac{s}{\alpha} y_{i} y_{j}\right),
\end{eqnarray*}
where $A:=a^2 \, m\,(2 m-1),\, B:=  b^2 \,k\,(2 k-1), \, H:= a\, b\, (m+k-1) \,(m+k) ,\,$\\ $O:=2\,a^2 \, m\,(m-~1),\,\, Z:= b^2 \,k\,(k-1),\,\,U:=a\, b (m+k-2) \,(m+k) $ are real constants.\\

Therefore, $g_{ij}$ are rational functions in $y$. Indeed, the coefficients of $\alpha_{ij}(x)$ are linear combination of the terms $s^{2m},\, s^{2k}$ and $ s^{n+k} $ which are rational  functions in $y$. Similarly, the coefficients of $b_{i}(x)\, b_{j}(x)$ are the linear combination of $s^{2m-2},\, s^{2k-2}$ and $ s^{k+m-2} $,  which are also rational  functions in $y$. The last term in $g_{ij}$ is the product of  rational functions in $y$, namely $\frac{1}{s\, \alpha} = \frac{1}{\beta}$, with the linear combination of the functions $s^{2m},\, s^{2k},\, s^{k+m} $.  It may be noted here that the term $ \frac{s}{\alpha}= \frac{\beta}{\alpha ^2}$ itself is rational in $y$. Since $g_{ij}$ is a rational function in $y$, we get \[\dot{\pa}_{i} \log(\eta(x,y))= 0= \pa_{i} \log(\eta(x,y)). \] 

\textbf{The $m$-th root metric.} \\

The Finsler mapping of the $m$-th root type is defined by, cf. \cite{$m$-th root}, 
\be \label{mthrootmetric def}
F(x,y):= \left( a_{i_{1} i_{2}...i_{m}}(x)\, y^{i_{1}}\, y^{i_{2}}\,... \,y^{i_{m}} \right)^{\frac{1}{m}}.
\ee

 In fact, $F(x,y)$ is a rational function in $y$ if and only if it is quadratic in $y$, that is $F$ is Riemannian. In addition, $F$ is positive definite when $m$ is even. It is clear that for $m\geq 3$ and $ n\geq 2$, the $m$-th root metric is an irrational function in $y$.

Let the polynomial \be \label{A mthrootmetric def}
 A :=  a_{i_{1} i_{2}...i_{m}}(x)\, y^{i_{1}}\, y^{i_{2}}\,... \,y^{i_{m}} = F^m .\ee
 
 Thereby, its fundamental metric tensor has the following components, 
 \[g_{ij}(x,y)= \frac{1}{m}A^{\frac{2}{m}-2}\, \left[ A \,\dot{\pa}_{i}\dot{\pa}_{j}A + \frac{2-m}{m} \,\dot{\pa}_{i}A\,\dot{\pa}_{j}A \right] .\] Thus, $F$ is an AR-Finsler metric. Indeed,  $g_{ij}(x,y)$ satisfies \eqref{gdef}  where $$\eta(x,y) = \frac{1}{m}\,A^{\frac{2}{m}-2}= \frac{1}{m}\,F^{2-2m}\,\, \text{ and }\,\,\, a_{ij}=A\,\dot{\pa}_{i}\dot{\pa}_{j}A - \frac{2-m}{m} \,\dot{\pa}_{i}A\,\dot{\pa}_{j}A .$$
Obviously, $\eta(x,y)$ is an irrational positive homogeneous function  of degree $(2-m)$ in $y$ and $a_{ij}(x,y)$ are rational  positive homogeneous functions of degree $(m-2)$ in $y$. 
By direct calculations, we get \[\dot{\pa}_{i} \log(\eta(x,y))=\frac{2(1-m)}{m^2}\, \frac{\dot{\pa}_{i}A(x,y)}{A(x,y)}= \frac{2(1-m)}{m}\, \frac{a_{i i_{2}...i_{m}}(x)\, y^{i_{2}}\,... \,y^{i_{m}}}{a_{i_{1} i_{2}...i_{m}}(x)\, y^{i_{1}}\, y^{i_{2}}... \,y^{i_{m}}} \] which are rational functions in $y$. Similarly,  
\[\pa_{i} \log(\eta(x,y))= \frac{2(1-m)}{m^2}\, \frac{\pa_{i}{a_{i_{1} i_{2}...i_{m}}}(x)\, y^{i_{1}}\, y^{i_{2}}\,... \,y^{i_{m}}}{a_{i_{1} i_{2}...i_{m}}(x)\, y^{i_{1}}\, y^{i_{2}}... \,y^{i_{m}}} \] are  rational functions in $y$. 
\begin{rem}
A special case of \eqref{mthrootmetric def} is the Finsler metric $F(x,y)= f(x)\,(y^{1} y^2 y^3)^{\frac{1}{3}}$ which defined in the conic domain  $\mathcal{D}= T\mathbb{R}^{3} - \{(x^{i},y^{i}) \in T\mathbb{R}^3\, |\,y^{i}\neq 0 \}$. It is an AR-Finsler manifold which does not admit a semi-concurrent vector field \cite[Remark 3.13]{semi}.
\end{rem}

\textbf{The generalized Kropina change $\widetilde{F}= \frac{F^{k+1}}{\beta ^k}$ of an $m$-th root metric $F$}. \\
Here,  $F$ is defined in \eqref{mthrootmetric def} and $k$ is an arbitrary positive real number cf. \cite{gen. Krop. change}. \\

One can check that, the fundamental metric tensor $\widetilde{g}_{ij}(x,y)$ satisfies \eqref{gdef}  with $$\widetilde{\eta}(x,y)= \frac{A^{\frac{2k+2-m}{m}}(x,y)}{\beta^{2k}(x,y)} =\frac{F^{2(k+1)-m}(x,y)}{\beta^{2k}(x,y)} \qquad\text{   and }$$ 
{\Small{\[\widetilde{a}_{ij}= \frac{k(2k+1)}{\beta^2}\,A\, b_{i}\, b_{j} + \frac{(k+1)}{m}\,\dot{\pa}_{i}\dot{\pa}_{j} A+ \frac{(k+1)(2k-m+2)}{m^2\,A}\, \dot{\pa}_{i}A\,\dot{\pa}_{j}A  - \frac{2k(k+1)}{m\, \beta}\,\left(b_{j}\,\dot{\pa}_{i}A+b_{i}\,\dot{\pa}_{j}A  \right) .\]}}
Therefore, $\widetilde{F}$ is an AR-Finsler metric.\\

Thereby, $\widetilde{\eta}(x,y)$ is an irrational positive homogeneous function of degree $(2-m)$ in $y$ and $\widetilde{a}_{ij}(x,y)$ are positive homogeneous of degree $(m-2)$ rational functions in $y$. 
By direct calculations, one can see that $\dot{\pa}_{i} \log(\widetilde{\eta}(x,y))$ and $\pa_{i} \log(\widetilde{\eta}(x,y))$ are rational functions in $y$. 

\begin{rem} In particular for $k=1$, $\widetilde{F}$ becomes the Kropina change of an $m$-th root metric which  has been studied in \cite{Krop. change}.
	\end{rem}

\textbf{ The extended $m$-th root metric.}\\

Now, let us introduce the following Finsler metric that can be considered as an extension of the $m$-th root metric. It is given by 
\be \label{m-root metric with  rational coeff}   
   F(x,y):= \left( \mu_{i_{1} i_{2}...i_{m}}(x,y)\, y^{i_{1}}\, y^{i_{2}}\,... \,y^{i_{m}} \right)^{\frac{1}{m}} ,\ee
where $ \mu_{i_{1} i_{2}...i_{m}}(x,y)$ symmetric in all indices, positive homogeneous of degree $(0)$ in $y$.
 In fact, $F(x,y)$ is a rational function in $y$ if and only if it is quadratic in $y$ and $ \mu_{i_{1} i_{2}...i_{m}}(x,y)$ are rational functions in $y$.  It is clear that for $m\geq 3$ and $ n\geq 2$, an extended $m$-th root metric \eqref{m-root metric with  rational coeff} is irrational function in $y$. Let \[ \mathcal{A}(x,y) =  \mu_{i_{1} i_{2}...i_{m}}(x,y)\, y^{i_{1}}\, y^{i_{2}}\,... \,y^{i_{m}} = F^m .\] Thus, the supporting element is given by
 \[ l_{i}= \dot{\pa}_{i} F = \dot{\pa}_{i}\mathcal{A}^{1/m}= \frac{1}{m}\, \mathcal{A}^{\frac{1-m}{m}} \,\dot{\pa}_{i}\mathcal{A} . \]
 The associated metric tensor has the following components 
\[ g_{ij}(x,y)=\mathcal{A}^{\frac{2-2m}{m}} \, \left[ \frac{1}{m}\, \mathcal{A}\,\dot{\pa}_{i}\dot{\pa}_{j} {\mathcal{A}}+ \frac{2-m}{m^{2}}\,\dot{\pa}_{i} {\mathcal{A}}\,\dot{\pa}_{j} {\mathcal{A}} \right] .\]
 
It should be noted that, when $\mu_{i_{1} i_{2}...i_{m}}(x,y)$  are  rational functions in $y$, then $g_{ij}(x,y)$ satisfies \eqref{gdef}  where 
\be \label{extended $m$-th root fund. metric}
\eta(x,y) =\mathcal{A}^{\frac{2-2m}{m}}= F^{2-2m}\,\, \text{ and }\,\, a_{ij}=\frac{1}{m}\,\mathcal{A}\, \dot{\pa}_{i}\dot{\pa}_{j} {\mathcal{A}}- \frac{2-m}{m^{2}}\,\dot{\pa}_{i} {\mathcal{A}}\,\dot{\pa}_{j} {\mathcal{A}}. \ee
 Indeed, in case of $ \mu_{i_{1} i_{2}...i_{m}}(x,y)$ are rational functions in $y$, the function $\mathcal{A}(x,y)$ is  rational in $y$. One can see that, $\eta(x,y)$ is an irrational  positive homogeneous of degree $(2-m)$ in $y$ and $a_{ij}(x,y)$ are rational positive homogeneous functions of degree $(m-2)$ in $y$. Consequently, $F$ is an AR-Finsler metric. \\
  
By direct calculations, we get 
\begin{eqnarray*}
\dot{\pa}_{i} \log(\eta(x,y))&=&(2-2m)\, \frac{\dot{\pa}_{i} {F(x,y)}}{F(x,y)} = \frac{(2-2m)}{m}\, \frac{\dot{\pa}_{i} {\mathcal{A}(x,y)}}{\mathcal{A}(x,y)} \\ &=& \frac{(2-2m)}{m}\, \frac{\left(m\,\mu_{i i_{2}...i_{m}}(x,y)\, y^{i_{2}}\,...\, y^{i_{m}} + y^{i_{1}}\, y^{i_{2}}\,...\, y^{i_{m}}\,\dot{\pa}_{i}\mu_{i_{1} i_{2}...i_{m}}(x,y)\right)}{\mu_{i_{1} i_{2}...i_{m}}(x,y)\, y^{i_{1}}\, y^{i_{2}}\,... \,y^{i_{m}}} 
\end{eqnarray*}
 which are rational functions in $y$. Similarly, ${\pa}_{i} \log(\eta(x,y))$ are rational functions in $y$. \\



\textbf{ The generalized Kropina change of the extended $m$-th root metric with  rational coefficients.}\\

Here, $F$ is given by \eqref{m-root metric with  rational coeff}, where $ \mu_{i_{1} i_{2}...i_{m}}(x,y)$ are rational functions in $y$. One can check that, $\widehat{F}:=\frac{F^{k+1}}{\beta^k}$ is an AR-Finsler metric. Indeed, its fundamental tensor $\widehat{g}_{ij}(x,y)$ satisfies \eqref{gdef}  with $$\widehat{\eta}(x,y) = \frac{\mathcal{A}^{\frac{2k+2-m}{m}}(x,y)}{\beta^{2k}(x,y)} =\frac{F^{2(k+1)-m }(x,y)}{\beta^{2k}(x,y) } \qquad\text{   and }$$ 

\begin{eqnarray*}
 \widehat{a}_{ij}&= &\frac{k(2k+1)}{\beta^2}\,\mathcal{A}\, b_{i}\, b_{j} + \frac{(k+1)}{m}\,\dot{\pa}_{i}\dot{\pa}_{j} \mathcal{A}+ \frac{(k+1)(2k-m+2)}{m^2\,\mathcal{A}}\, \dot{\pa}_{i}\mathcal{A}\,\dot{\pa}_{j}\mathcal{A} \\
 && - \frac{2k(k+1)}{m\, \beta}\,\left(b_{j}\,\dot{\pa}_{i}\mathcal{A}+b_{i}\,\dot{\pa}_{j}\mathcal{A} \right) .
 \end{eqnarray*}

Thereby, $\widehat{\eta}(x,y)$ is an irrational positive homogeneous function of degree $(2-m)$ in $y$ and $\widehat{a}_{ij}(x,y)$ are  rational functions in $y$ and positive homogeneous of degree $(m-2)$ in $y$. 
Moreover, one can show that $\dot{\pa}_{i} \log(\eta(x,y))$ and $\pa_{i} \log(\eta(x,y))$ are rational functions in $y$. 
\begin{rem} In particular for $k=1$, $\widehat{F}$ becomes the Kropina change of an extended $m$-th root metric with rational coefficients.
\end{rem}
\begin{rem}
In the above mentioned examples, we observe the following:\\

$(i)$ 
The generalized Kropina change of an $m$-th root metric or extended $m$-th root metric with rational coefficients  preserves the almost rationality of the Finsler metric.\\ 

$(ii)$ The  geometric objects associated to the AR-Finsler metrics, namely,  $I_{k},\,\,G^{i},\,\, N^{i}_{j},\,\, G^{i}_{jk}$ 

$G^{i}_{jkl},\,\, D^{i}_{jkl},\,\, L^{i}_{jkl},\,\,J_{k},\,\, R^{i}_{j},\,\, \chi_{i},\,\, W^{i}_{j} $ and the $S$-curvature are rational functions in $y$. It follows from  the results of \S 3 and  these AR-metrics satisfy the property that  ${\pa}_{i} \log(\eta(x,y))$ are rational functions in $y$.
\end{rem} 



\bibliographystyle{plain}

\end{document}